\documentclass[a4paper,10pt]{amsart}

\usepackage{epsfig}
\usepackage{amsthm,amsfonts}
\usepackage{amssymb,graphicx,color}
\usepackage{float}
\graphicspath{ {Images/} }
\usepackage[labelfont={bf},textfont={it}]{caption}
\usepackage[labelfont={bf, it}]{subcaption}
\usepackage[all]{xy}
\usepackage{verbatim}
\usepackage{hyperref}
\usepackage{enumitem}

\newtheorem{theorem}{Theorem}[section]
\newtheorem*{theorem*}{Theorem}
\newtheorem{lemma}[theorem]{Lemma}
\newtheorem{corollary}[theorem]{Corollary}

\newtheorem{definition}[theorem]{Definition}
\newtheorem{conjecture}{Conjecture}
\newtheorem{claim}{Claim}
\newtheorem{example}[theorem]{Example}
\newtheorem{remark}[theorem]{Remark}

\newtheorem{innercustomthm}{Theorem}
\newenvironment{customthm}[1]
  {\renewcommand\theinnercustomthm{#1}\innercustomthm}
  {\endinnercustomthm}

\newcommand{\R}{\mathbb{R}}

\newcommand{\N}{\mathbb{N}}

\begin{document}

\title[On characterization of smoothness of complex analytic sets]
{On characterization of smoothness of complex analytic sets}

\author[A. Fernandes]{Alexandre Fernandes}
\author[J. E. Sampaio]{Jos\'e Edson Sampaio}
\address[Alexandre Fernandes and Jos\'e Edson Sampaio]{    
              Departamento de Matem\'atica, Universidade Federal do Cear\'a,
	      Rua Campus do Pici, s/n, Bloco 914, Pici, 60440-900, 
	      Fortaleza-CE, Brazil. \newline  
              E-mail: {\tt alex@mat.ufc.br}\newline  
              E-mail: {\tt edsonsampaio@mat.ufc.br}
}

\thanks{The first named author was partially supported by CNPq-Brazil grant 304221/2017-9.
The second author was partially supported by CNPq-Brazil grant 303811/2018-8, by the ERCEA 615655 NMST Consolidator Grant and also by the Basque Government through the BERC 2018-2021 program and Gobierno Vasco Grant IT1094-16, by the Spanish Ministry of Science, Innovation and Universities: BCAM Severo Ochoa accreditation SEV-2017-0718.
This study was financed in part by the CAPES-BRASIL Finance Code 001 and PRONEX-FUNCAP/CNPq/Brasil.}


\begin{abstract}

The paper is devoted to metric properties of singularities. We investigate the relations among topology, metric properties and smoothness.  In particular, we present some higher dimensional analogous of Mumford's theorem on smoothness of normal surfaces. For example, we prove that a complex analytic set, with an isolated singularity at $0$, is smooth at $0$ if and only if it is locally metrically conical at $0$ and its link at $0$ is a homotopy sphere.
\end{abstract}

\maketitle


\section{Introduction}
A fundamental result in topology of singularities was proved by  D. Mumford in \cite{Mumford:1961}. It was shown that {\it 2-dimensional normal complex algebraic germs with simply connected link must be smooth. In particular, the link of such a germ is diffeomorphic to the 3-dimensional Euclidean sphere.}

E. Brieskorn \cite{Brieskorn:1966,Brieskorn:1966b}, D. Prill \cite{Prill:1967}, A'Campo \cite{Acampo:1973}, L\^e D. T. \cite{Le:1973} and others gave important contributions devoted to the relations of smoothness of a complex algebraic set at $0$ and the topology of its link at $0$. For instance, E. Brieskorn \cite{Brieskorn:1966,Brieskorn:1966b} exhibited examples of isolated singular (not smooth) germs of complex algebraic sets in higher dimension with link homeomorphic (not diffeomorphic) to Euclidean spheres. In other words, it was shown that Theorem of Mumford, mentioned above for 2-dimensional germs, does not work in higher dimensions without additional assumptions. As another reference of works in the same line, we also have D. Prill \cite{Prill:1967} where it was proved that {\it any d-dimensional complex algebraic cone with trivial i-th homotopy group of the link, for all $0\leq i \leq 2d-2$, must be a linear affine subspace.} One more result of the type of  Theorem of Mumford was obtained separately by N. A'Campo in \cite{Acampo:1973} and L\^e D. T. in \cite{Le:1973}, which can be stated as follows:
{\it
Let $(X,x_0)$ be a codimension one germ of complex analytic subset of $\mathbb{C}^n$. If there is a homeomorphism $\varphi\colon(\mathbb{C}^n,X,x_0)\to (\mathbb{C}^n,\mathbb{C}^{n-1}\times\{0\},x_0)$, then $X$ is smooth at $x_0$.
}

Recently, there were presented two different approaches for a type of Theorem of Mumford in higher dimension: one from contact geometry point of view and the other one  from Lipschitz geometry point of view. From the  contact geometry point of view, M. McLean in \cite{McLean:2016} showed that: if a complex analytic set $A$ has a normal isolated singularity at $0$ and its link at $0$ (with its canonical contact structure) is contactomorphic to the link of $\mathbb{C}^3$ (the standard contact sphere), then $A$ is smooth at $0$. T. de Fernex and Y.-C. Tu in \cite{FernexT:2017} also presented some results on characterization of smoothness with the same approach as M. McLean.

From the Lipschitz geometry point of view, it was proved by the authors of this paper jointly with L. Birbrair and L\^e D. T. in \cite{BirbrairFLS:2016} that: any germ of complex analytic set which is  subanalytically Lipschitz regular must be smooth; and the second named author of this paper showed in \cite{Sampaio:2016} the same result holds true without the "subanalytic" condition in the statement above.

Here we study possible analogs of Theorem of Mumford in higher dimension from Lipschitz geometry point of view. It is possible to state that complex analytic sets are smooth with additional minimal assumptions on the geometry of these sets. In particular, that their links are diffeomorphic to Euclidean spheres. The main results of this paper rely on a stability theorem of fundamental groups for Hausdorff limit of {\it Lipschitz normally embedded sets}. Thus, before the presentation of the main result of this paper, we need to introduce the notion of {\it Lipschitz normally embedded sets}.

Let us remind that, for path connected subsets $A\subset\R^n$, the {\it inner distance} of $A$ is defined by: $d_A(p,q)$ is the infimum of the lengths of paths connecting $p$ to $q$ on $A$.  In this way, we say that a path connected subset $A\subset\R^n$ is {\it Lipschitz normally embedded} (LNE) if there exists a constant $C\geq 1$ such that $$ d_{A}(p,q)\leq C|p-q| \ \mbox{for all} \ p,q\in A.$$ In order to be more precise, in that case we say that $A$ is $C$-LNE.

From now on, we start to present the main achievements of this paper. The first result is the following theorem which is presented in Section \ref{sec:key_results}; it is our stability theorem of fundamental groups for Hausdorff limit of C-LNE sets.

\begin{customthm}{\ref*{Fundamental_thm}}
	Let $X_0,X_1\subset\R^n$ be compact subanalytic $C$-LNE subsets. Let $\epsilon_0(X_i)>0$ be the biggest real number such that: any loop on $X_i$ with length smaller than $\epsilon_0$ is homotopically trivial ($i=0,1$). If the Hausdorff distance between $X_0$ and $X_1$, $dist_{H}(X_0,X_1)$, is smaller than $\displaystyle \frac{\epsilon_0(X_0)}{20C^2}$, then there are based points $x_0\in X_0$, $x_1\in X_1$ and an epimorphism from the fundamental group $\pi_1(X_1,x_1)$ onto the fundamental group $\pi_1(X_0,x_0)$.  Moreover, if $dist_{H}(X_0,X_1)<\frac{1}{20C^2}\min \{\varepsilon_0(X_0),\varepsilon_0(X_1)\}$ we can take the above epimorphism being an isomorphism.
\end{customthm}

As an important consequence, we obtain that if $X$ is a subanalytic set which is LNE at $0$ and has connected link at $0$ then there is an epimorphism from the fundamental group of the link of $X$ at $0$ onto the fundamental group of the link of its tangent cone at $0$ (see Corollary \ref{Fundamental_thm_local}).

Next result, which is presented in Section \ref{sec:main_results}, is concerning to smoothness of complex algebraic sets; It comes as a consequence of Theorem \ref{Fundamental_thm}.


\begin{customthm}{\ref*{main_result}}
	Let $X\subset\mathbb{C}^n$ be a complex analytic set of dimension $k>1$ at $0\in X$. If $X$ is $C$-LNE at $0$ and its link at $0$ is simply connected, then the following conditions are mutually equivalent:
	\begin{enumerate}
		\item $X$ is $2k$-homology manifold and locally linearly contractible at $0$ (see Definitions \ref{def:homolgy_manifold} and \ref{def:llc});
		\item $X$ has no choking cycles at $0$ and its link at $0$ has trivial $i$-th integral homology group for all $2\leq i\leq 2k-2$ (see Definition \ref{def:choking_cycles});
		\item The link of the tangent cone of $X$ at $0$ has trivial $i$-th integral homology group for all $2\leq i\leq 2k-2$;
		\item $X$ is smooth at $0$.
	\end{enumerate}
    In particular, if any of the above items holds true, we get that $X_t:=(\frac{1}{t} X)\cap \mathbb S^{2m-1}$ is diffeomorphic to $\mathbb S^{2k-1}$ for all small enough $t>0$.
\end{customthm}

In Section \ref{sec:main_results} is also presented a proof of  Theorem \ref{main_result} and several consequences of this theorem.
In order to present the main consequence, let us remind that a subanalytic set $A\subset\R^n$ is {\it metrically conical (at $0$)} if there exist $\varepsilon>0$ and a bi-Lipschitz homeomorphism $h\colon A\cap \overline{B_{\varepsilon}(0)}\to Cone(A\cap \mathbb{S}^{n-1}_{\varepsilon})$ such that $\|h(x)\|=\|x\|$ for all $ A\cap \overline{B_{\varepsilon}(0)}$ and $A$ is {\it locally metrically conical (at $0$)} if for all $v\in C(A,0)\cap \mathbb{S}^{n-1}$ there exists $\delta>0$ such that $Cone_1(B_{\delta}(v))\cap A$ is metrically conical (at $0$), where $\overline{B_{\varepsilon}(x_0)}=\{x\in \R^n;\|x-x_0\|\leq \varepsilon\}$, $\mathbb{S}^{n-1}_{\varepsilon}=\{x\in \R^n;\|x\|= \varepsilon\}$ and $Cone_1(X)=\{tx;t\in [0,1]$ and $x\in X\}$.

\begin{customthm}{\ref*{main_consequence}}
	Let $X\subset\mathbb{C}^n$ be a complex analytic set of dimension $k$ with isolated singularity at $0$. Then, $X$ is smooth at $0$ if and only if $X$ is locally metrically conical at $0$ and its link at $0$ is $(2k-2)$-connected.
\end{customthm}

Theorem \ref{main_consequence} gives a partial answer for Neumann-Pichon's conjecture about metrically conical sets, which says the following:

\begin{conjecture}
 Let $X\subset\mathbb{C}^n$ be a complex analytic set with isolated singularity at $0$. Then,  $X$ is metrically conical at $0$ if and only if  $X$ is locally metrically conical at $0$.
\end{conjecture}

%


\bigskip

\noindent{\bf Acknowledgements}. Alexandre Fernandes wishes to express his gratitude to BCAM for its hospitality and support during the final part of preparation of this paper. We wish to thank Lev Birbrair for his suggestions on the introduction of this paper. We wish to thank Pepe Seade for his interest in this paper.

\section{Preliminaries}

All the subsets of $\R^n$ or $\mathbb{C}^n$ considered in the paper are supposed to be equipped with the Euclidean distance. When we consider the inner distance, it is clearly emphasized.

Given a path connected subset $X\subset\R^n$, the
\emph{inner distance}  on $X$  is defined as follows: given two points $x_1,x_2\in X$, $d_X(x_1,x_2)$  is the infimum of the lengths of paths on $X$ connecting $x_1$ to $x_2$.

\begin{definition}\label{def:lne}
Let $X\subset\R^N$ be a subset. We say that $X$ is {\bf Lipschitz normally embedded (LNE)} if there exists a constant $C\geq 1$ such that $d_{X}(x_1,x_2)\leq C\|x_1-x_2 \|$, for all pair of points $x_1,x_2\in X$. In this case, we also say that $X$ is $C$-LNE. We say that $X$ is {\bf LNE at $p\in \R^N$}  if there exists an open neighborhood $U$ of $p$ such that $X\cap U$ is LNE. We say that $X$ is {\bf LNE at infinity} if there exists a compact subset $K\subset\mathbb{R}^N$ such that $X\setminus K$ is LNE.
\end{definition}

For instance, considering the real (resp. complex) cusp $x^2=y^3$, in $\R^2$ (resp. in $\mathbb{C}^2$), one can see that this set is not LNE.

Let us remark that the definition of LNE set was introduced by L. Birbrair and T. Mostowski \cite{BirbrairM:2000}, where they just call it by normally embedded set.

In the paper \cite{MendesS:2021} was introduced the following related definition:

\begin{definition}\label{def:link_lne}
Let $X\subset\R^N$ be a subset, $p\in \R^N$ and $X_t:=[\frac{1}{t}(X-p)]\cap \mathbb{S}^{N-1}$ ($X_t:=(\frac{1}{t}X)\cap \mathbb{S}^{N-1}$) for all $t>0$. We say that $X$ is {\bf link Lipschitz normally embedded (LLNE) at $p$} (resp. LLNE at infinity) if there exist constants $C\geq 1$ and $\delta >0$ (resp. $R>0$) such that $X_{t}$ is $C$-LNE for all $0<t\leq \delta$ (resp. $t\geq R$). In this case, we also say  that $X$ is $C$-LLNE at $p$ (resp. $C$-LLNE at infinity).
\end{definition}

The second named author of this paper joint with R. Mendes proved in \cite[Theorem 1.1]{MendesS:2021} that LNE and LLNE notions are equivalent for subanalytic set-germs with connected links.

\begin{example}
Let $X\subset\mathbb{C}^m$ be a complex analytic set with isolated singularity at $0\in X$. Then, for all small enough $t>0$, the set
$X_t=(\frac{1}{t}X)\cap\mathbb{S}^{2m-1}$
is a compact and smooth manifold. In particular, there exists $C_t\geq 1$ such that $d_{X_t}(x,y)\leq C_t \|x-y\|$ for all $x,y\in X_t$. Thus, when such $C_t$ does not depend on $t$, we are just in the case that $X$ is LLNE at $0$.
\end{example}

\begin{example}
Let $X\subset\R^m$ be a subanalytic subset. If there exists a subanalytic bi-Lipschitz homeomorphism $\psi\colon (X,0)\to (\R^d,0)$ ($d>1$) then $X$ is LLNE at $0$.
\end{example}
\begin{example}
Let $X\subset\R^m$ be a subset. If $X$ is an $n$-dimensional smooth submanifold ($n>1$) then it is LLNE at any point $x\in X$.
\end{example}




\begin{definition}\label{def:tg_cone}
Let $X\subset \R^{n}$ be a locally compact subset and let $p\in X$.
We say that $Z\subset \R^{n}$ is {\bf a tangent cone} of $X$ at $p$ (at infinity) if there is a sequence $\{t_j\}_{j\in \N}$ of positive real numbers such that $\lim\limits_{j\to \infty} t_j=0$ (resp. $\lim\limits_{j\to \infty} t_j=+\infty$) and the sequence of sets $\frac{1}{t_j}(X-p)$ (resp. $\frac{1}{t_j}X$) converges locally in the Hausdorff distance to $Z$. When $X$ has a unique tangent cone at $p$ (resp. at infinity), we denote it by $C(X,p)$ (resp. $C_{\infty}(X)$) and we call $C(X,p)$ (resp. $C_{\infty}(X)$) {\bf the tangent cone of $X$ at $p$} (resp. {\bf the tangent cone of $X$ at infinity}).
\end{definition}
\begin{remark}
Let $X\subset \mathbb{C}^n$ be a complex analytic (resp. algebraic) set and $x_0\in X$. In this case, $C(X,x_0)$ (resp. $C_{\infty}(X)$) is the zero set of a set of complex homogeneous polynomials, we can see this in \cite[Theorem 4D]{Whitney:1972} (resp. \cite[Theorem 1.1]{LeP:2016}). In particular, $C(X,x_0)$ (resp. $C_{\infty}(X)$) is the union of complex lines passing through the origin $0\in\mathbb{C}^n$.
\end{remark}

\begin{definition}
A subanalytic set $A\subset\R^n$ is {\bf metrically conical (at $0$)} if there exist $\varepsilon>0$ and a bi-Lipschitz homeomorphism $h\colon A\cap \overline{B_{\varepsilon}(0)}\to Cone(A\cap \mathbb{S}^{n-1}_{\varepsilon})$ such that $\|h(x)\|=\|x\|$ for all $ A\cap \overline{B_{\varepsilon}(0)}$ and $A$ is {\bf locally metrically conical (at $0$)} if for all $v\in C(A,0)\cap \mathbb{S}^{n-1}$ there exists $\delta>0$ such that $Cone(B_{\delta}(v))\cap A$ is metrically conical (at $0$), where $\overline{B_{\varepsilon}(x_0)}=\{x\in \R^n;\|x-x_0\|\leq \varepsilon\}$, $\mathbb{S}^{n-1}_{\varepsilon}=\{x\in \R^n;\|x\|= \varepsilon\}$ and $Cone(X)=\{tx;t\in [0,1]$ and $x\in X\}$.
\end{definition}

There also is the notion of a set to be (locally) metrically conical with respect to inner distance, but we do not approach it here.

An algebraic subset of $\mathbb{C}^m$ is called a \emph{complex cone} if it is a union of one-dimensional linear subspaces of $\mathbb{C}^m$. Next result was proved by D. Prill in \cite{Prill:1967}.

\begin{lemma}[\cite{Prill:1967}, Theorem, p. 180] \label{prill}
Let $V\subset \mathbb{C}^m$ be a complex cone of dimension k with
$\pi_j(V\setminus \{0\})=0$ for $0\leq j\leq 2k-2$. Then $V$ is a linear subspace of $\mathbb{C}^m$.
\end{lemma}

\begin{definition}\label{def:ahlfors}
A metric space is called {\bf locally Ahlfors n-regular} if it has Hausdorff dimension $n$ and if for every compact subset $K$ of the space there exist numbers $r_K > 0$ and $C_K \geq 1$ such that
$$
C_K^{-1} r^n \leq  \mathcal{H}_n(B_r(x)) \leq  C_Kr^n,
$$
for every metric ball $B_r(x)$ in the space, with center $x\in K$ and radius $r < r_K$, where $\mathcal{H}_n(A)$ denotes the $n$-dimensional Hausdorff measure of $A$.
\end{definition}

\begin{definition}\label{def:llc}
A metric space is said {\bf locally linearly contractible (LLC)} if for every compact set $K$ in the space there exist numbers $r_K > 0$ and $C_K \geq  1$ such that every metric ball $B_r(x)$ in the space, with center $x\in K$ and radius $r < r_K$, contracts to a point inside $B_{C_Kr}(x)$.
\end{definition}

\begin{definition}\label{def:homolgy_manifold}
A separable and metrizable space $X$ is a {\bf homology $n$-manifold}, $n \geq  2$, if $X$ is finite dimensional, locally compact and locally contractible, and if the integral homology groups satisfy
$$
H_*(X; X \setminus \{x\}) = H_*(\R^n; \R^n\setminus \{0\})
$$
for each $x\in X$.
\end{definition}

\begin{definition}
We say that $X$ is an LLC homology $n$-manifold at $p\in X$ if there is an open neighborhood $U\subset X$ of $p$ such that $X\cap U$ is a homology $n$-manifold and LLC.
\end{definition}

\begin{definition}
A metric space is said {\bf $n$-rectifiable} if it can be covered, up to a set of Hausdorff $n$-measure zero, by a countable number of Lipschitz images of subsets of $\R^n$. We call a metric space {\bf metrically n-dimensional} if it is $n$-rectifiable and locally Ahlfors $n$-regular.
\end{definition}

\begin{lemma}[Proposition 4.1 in \cite{Heinonen:2011}]\label{heinonen}
Let $X$ be a metrically $n$-dimen\-si\-o\-nal and locally linearly contractible homology $n$-manifold, and let $p $ be a point in $X$. Then $C(X,p)$ is a homology $n$-manifold.
\end{lemma}

\begin{definition}\label{def:choking_cycles}
We say that $X\subset \R^N$ has an $n$-dimensional family of choking cycles at $0\in X$ (resp. at infinity) if:
\begin{itemize}
 \item [a)] There exist constants $\lambda\geq 1$, $\delta>0$ (resp. $R>0$);
 \item [b)] There exists a family of $n$-dimensional integral singular cycles $\{\xi_t\}_{0<t<\delta}$ (resp. $\{\xi_t\}_{t>R}$) such that $\xi_t$ is in the set
 $
 X_t:=(\frac{1}{t} X)\cap \mathbb{S}^{N-1}
 $
 and $diam |\xi_t|\to 0$ as $t\to 0$ (resp. $t\to +\infty$);
 \item [c)] if $\{\eta_t\}$ is a family of $(n+1)$-dimensional integral singular chains such that each $\eta_t$ is in $X_t$ and $\partial \eta_t=\xi_t$ for all $0<t<\delta$ (resp. $t>R$) then $\liminf diam |\eta_t|>0$.
\end{itemize}

\end{definition}
Let us mention that this object first arose in the paper \cite{BirbrairFGO} where it was proved the existence of choking horns (choking cycles where $\xi_t$ are spheres) is an obstruction to the germ being bi-Lipschitz homeomorphic to a cone.

When $X\subset \R^N$ is a subanalytic (resp. semialgebraic) subset, we denote by $L_p(X)$ (resp. $L_{\infty}(X)$) to be the link of $X$ at $p$ (resp. infinity).



\section{LNE sets and stability of fundamental groups }\label{sec:key_results}
\begin{remark}
Let $Y$ be a compact subanalytic set. Then, it is locally simply connected and, moreover, since it is compact, there is $\varepsilon _0>0$ such that every closed curve in it of length smaller than $\varepsilon_0$ is homotopy equivalent to a constant. We detote by $\varepsilon_0(Y)$ the supremum of all such $\varepsilon_0$ above.
\end{remark}

\begin{theorem}\label{Fundamental_thm}
Let $X_0,X_1$ be two $C$-LNE compact subanalytic subsets of $\mathbb R^m$. If $dist_{H}(X_1,X_0)<\frac{\varepsilon_0(X_0)}{20C^2}$ then there exist based points $y_0\in X_0$, $x_0\in X_1$ and an epimorphism $h\colon \pi_1(X_1,x_0)\to \pi_1(X_0,y_0)$. Moreover, if
$$dist_{H}(X_1,X_0)<\frac{1}{20C^2}\min \{\varepsilon_0(X_0),\varepsilon_0(X_1)\}$$
we can take $h$ being an isomorphism.
\end{theorem}
\begin{proof}
Take $\varepsilon  = \frac{\varepsilon_0(X_0)}{20C^2}$. 
Let $x_0\in X_1$ and $y_0\in X_0$ such that $\|x_0-y_0\|<\varepsilon$.

We define $h\colon \pi_1(X_1, x_0)\to \pi_1(X_0,y_0)$ in the following way:

Let $\omega \in \pi_1 (X_1,x_0)$. Since $X_1$ is a compact subanalytic set, $\omega$ contains some rectifiable closed curve $\alpha$. Choose now an partition $P=\{t_0=0,t_1, ... ,t_{n}=1\}$ of $[0,1]$ such that for all $i = 0, ... ,n - 1$, one has $length(\alpha|_{[t_i,t_{i+1}]}) <3C\varepsilon $.
For each $i=0, ... ,n$ choose $y_i\in X_0$ satisfying $\|x_i-y_i\| < \varepsilon $, where $x_i = \alpha(t_i)$. Moreover, we assume that $y_{i}=y_0$ whenever $x_i=x_0$. Then for each $i \in \{0, ... ,n - 1\}$, we connect $y_i$ and $y_{i+1}$ by a geodesic $\tilde\alpha_i$ in $X_0$.
This yields a broken geodesic $\tilde\alpha \colon [0,1]\to X_0$ with base point $y_0$. Thus, we set $h([\alpha])=[\tilde\alpha]$.

\begin{claim}\label{ht_welldefined_one}
$[\tilde\alpha]$ does not depend on the choice of the geodesics connecting $y_i$ and $y_{i+1}$.
\end{claim}
\begin{proof}
In fact, if $\tilde\beta_i$ is another geodesic in $X_0$ connecting $y_i$ and $y_{i+1}$, we have that
$$
\begin{array}{ll}
length(\tilde\alpha_i)=length(\tilde\beta_i)&=d_{X_0}(y_i,y_{i+1})\\
           &\leq C\|y_i-y_{i+1}\|\\
		   &\leq C(\|y_i-x_i\|+\|x_i-x_{i+1}\|+ \|x_{i+1}-y_{i+1}\|)\\
		   &< C(\varepsilon +length(\alpha|_{[t_i,t_{i+1}]})+\varepsilon)\\
		   &\leq C(\varepsilon +3C\varepsilon+\varepsilon)\\
		   &\leq 5C^2\varepsilon.
\end{array}
$$
Thus $length(\tilde\alpha_i\cdot (\tilde\beta_i)^{-1})<10C^2\varepsilon<\varepsilon_0$, which implies that $\tilde\alpha_i\cdot (\tilde\beta_i)^{-1}$ is homotopy equivalent to a constant for any $i \in \{0, ... ,n - 1\}$.  This implies that $[\tilde\alpha ]=[\tilde\beta]$, where $ \tilde\beta$ is the broken geodesic given by the $\tilde\beta_i$'s.
\end{proof}
\begin{claim}\label{ht_welldefined_two}
$[\tilde\alpha]$ does not depend on the choice of the $y_i$'s.
\end{claim}
\begin{proof}
Assume that for some $i\in \{1,...,n-1\}$, we choose another $y_i'$. We consider a geodesic $\tilde \alpha_{i-1}'$ (resp. $\tilde \alpha_{i}'$) connecting $y_{i-1}$ and $y_i'$ (resp. $y_{i}'$ and $y_{i+1}$). Thus we consider $\tilde\alpha'$ the broken geodesic given by the $\alpha_i$'s with $\tilde \alpha_{i-1}'$ and $\tilde \alpha_{i}'$ resp. instead of $\tilde \alpha_{i-1}$ and $\tilde \alpha_{i}$. We want to prove that $[\tilde\alpha]=[\tilde\alpha']$. In order to do this, we have that
$$
\begin{array}{ll}
length(\tilde\alpha_{j})&\leq C\|y_{j-1}-y_j\|\\
                        &\leq C(\|y_{j-1}-x_{j-1}\|+\|x_{j-1}-x_j\|+\|x_j-y_j\|)\\
                        &<5C^2\varepsilon,
\end{array}
$$
and similarly
$$
length(\tilde\alpha_{j}')<5C^2\varepsilon,
$$
for $j=i-1,i$. Thus,
$$length((\tilde\alpha_i\cdot\tilde\alpha_{i-1})\cdot (\tilde\alpha_i'\cdot\tilde\alpha_{i-1}')^{-1})<20C^2\varepsilon=\varepsilon_0(X_0),$$
which implies that $(\tilde\alpha_i\cdot\tilde\alpha_{i-1})\cdot (\tilde\alpha_i'\cdot\tilde\alpha_{i-1}')^{-1}$ is homotopy equivalent to a constant and then $[\tilde\alpha]=[\tilde\alpha']$.
\end{proof}
\begin{claim}\label{ht_welldefined_three}
$[\tilde\alpha]$ does not depend on the choice of the $x_i$'s.
\end{claim}
\begin{proof}
Let us consider $t_{i}<t'<t_{i+1}$ such that $x'=\alpha(t')\not\in \{x_0,....,x_n\}$. Let $y'\in X_0$ such that $\|y'-x'\|<\varepsilon$. We consider $\tilde\gamma_i$ (resp. $\tilde \gamma_i'$) a geodesic connecting $y_{i}$ and $y'$ (resp. $y'$ and $y_{i+1}$). Thus, we consider $\tilde\alpha'$ the broken geodesic given by the $\alpha_j$'s with $\tilde \gamma_{i}'\cdot \tilde\gamma_{i}$ instead of $\tilde \alpha_{i}$.

Using similar estimates as before we obtain the following
$$
length(\tilde\alpha_i\cdot (\tilde\gamma_i'\cdot\tilde\gamma_{i})^{-1})<15C^2\varepsilon<\varepsilon_0,
$$
which implies that $\tilde\alpha_i\cdot (\tilde\gamma_i'\cdot\tilde\gamma_{i})^{-1}$ is homotopy equivalent to a constant and then $[\tilde\alpha]=[\tilde\alpha']$.

Applying inductively the above argument, we see that $[\tilde\alpha]$ does not depend on the choice of the $x_i$'s.
\end{proof}

\begin{claim}\label{ht_welldefined_four}
$[\tilde\alpha]$ does not depend on the choice of the $\alpha\in \omega$.
\end{claim}
\begin{proof}
Let $\beta$ be another curve in $\omega$ such that $length(\beta)<+\infty$, then there is a homotopy $H\colon[0,1]\times [0,1] \to X_1$ between $\alpha$ and $\beta$. 

We can use uniform continuity of $H$ to show that there is a partition $\{r_0=0,r_1, ... ,r_{q}=1\}$ of $[0,1]$ that refines $P$ such that $\|H(r_i,\tau)-H(r_{i+1},\tau)\| <\varepsilon$ and $\|H(s,r_i)-H(s,r_{i+1})\|<\varepsilon$ for all $\tau,s\in [0,1]$ and $i\in\{0,1,...,q-1\}$. Now, for $i,j\in \{0,1,...,q\}$ we choose $y_{ij}\in X_0$ such that $\|y_{ij}-H(r_i,r_j)\|<\varepsilon$. Moreover, we assume that $y_{ij}=y_0$ whenever $H(r_i,r_j)=x_0$. Thus, for each $i\in \{0,1,...,q\}$ we define $\tilde\gamma_i\colon [0,1]\to X_0$ to be a broken geodesic such that $\tilde\gamma_{ij}=\tilde\gamma_i|_{[r_j,r_{j+1}]}$ is a geodesic connecting $y_{ij}$ and $y_{i(j+1)}$ for $j\in \{0,1,...,q-1\}$. However, by before claims, we can assume that $\tilde \gamma_0=\tilde\alpha$ and $\tilde \gamma_q=\tilde\beta$. Therefore, in order to show that $[\alpha]=[\beta]$, it is enough to show that $\tilde\gamma_i$ is homotopy equivalent to $\tilde\gamma_{i+1}$ for each $i\in\{0,1,...,q-1\}$. Thus, let $i\in\{0,1,...,q-1\}$ fixed, and for each $j\in\{0,1,...,q\}$, let $\sigma_j$ be a geodesic on $X_0$ connecting $y_{ij}$ and $y_{(i+1)j}$. By using similar estimates as before once again, for each $j\in\{0,1,...,q-1\}$ we obtain
$$
length(\tilde\gamma_{ij}\cdot (\sigma_{j+1}\cdot(\tilde\gamma_{(i+1)j})^{-1}\cdot (\sigma_j)^{-1}))<20C^2\varepsilon=\varepsilon_0(X_0).
$$
Then, $\tilde\gamma_{ij}$ and $\sigma_{j+1}\cdot(\tilde\gamma_{(i+1)j})^{-1}\cdot (\sigma_j)^{-1}$ are homotopy equivalent for each $j\in\{0,1,...,q-1\}$, which implies that $\tilde\gamma_{i}$ and $\tilde\gamma_{i+1}$ are homotopy equivalent.
\end{proof}

Therefore $h$ is well defined.

\begin{claim}\label{ht_homomorphism}
$h$ is a homomorphism.
\end{claim}
\begin{proof}
Let $\alpha$ and $\beta$ be two loops on $X_1$ based at $x_0$ such that
$$length(\alpha),length(\beta)<+\infty.$$
If $\gamma=\alpha\cdot \beta$ then a partition $P=\{t_0=0,t_1, ... ,t_{n}=1\}$ of $[0,1]$ such that $1/2=t_k$ for some $k\in \{1,...,n-1\}$ and for all $i = 0, ... ,n - 1$ one has $length(\gamma|_{[t_i,t_{i+1}]}) <3C\varepsilon $ implies that $length(\alpha|_{[2t_i-1,2t_{i+1}-1]}) <3C\varepsilon $ for all $i\in \{0,...,k-1\}$ and $length(\beta|_{[2t_i,2t_{i+1}]}) <3C\varepsilon $ for all $i\in \{k,...,n-1\}$. By before claims, we obtain
$$
h([\gamma])=[\tilde\gamma]=[\tilde\alpha\cdot \tilde \beta]=[\tilde\alpha]\cdot [\tilde \beta]=h([\alpha])\cdot h([\beta]).
$$
\end{proof}

\begin{claim}\label{ht_surjective}
$h$ is surjective.
\end{claim}
\begin{proof}
Let $\theta \in \pi_1 (X_0,y_0)$. Since $X_0$ is a compact subanalytic set, $\theta$ contains some rectifiable closed curve $\beta$. Choose now an partition $P=\{t_0=0,t_1, ... ,t_{n}=1\}$ of $[0,1]$ such that for all $i = 0, ... ,n - 1$, one has $length(\beta|_{[t_i,t_{i+1}]}) <\varepsilon $.
For each $i=0, ... ,n$ choose $x_i\in X_1$ satisfying $\|x_i-y_i\| < \varepsilon $ with $x_n=x_0$ and we connect those points by minimizing geodesics in $X_1$, where $y_i = \beta(t_i)$.
This yields a broken geodesic $\alpha \colon [0,1]\to X_1$ with base point $x_0$ such that $length(\alpha|_{[t_i,t_{i+1}]}) <3C\varepsilon $.
Since $[\tilde\alpha]=h([\alpha])$ does not depend on the choice of the points in $X_0$, we can choose these points to be $y_0,...,y_n$. Thus, by using estimates as above, we obtain that $\tilde\alpha$ and $\beta$ are homotopy equivalent. Therefore $h([\gamma])=[\beta]=\theta$.
\end{proof}
This finishes the proof of the first part of Theorem \ref{Fundamental_thm}.
For the second part, in a similar way, we can find an epimorphism $\tilde h\colon \pi_1(X_0)\to \pi_1(X_1)$ when  $dist_{H}(X_1,X_0)<\frac{1}{20C^2}\min \{\varepsilon_0(X_0),\varepsilon_0(X_1)\}$. Then, it follows from the way that we built $h$ and $\tilde h$ that $\tilde h\circ h=id$.
\end{proof}

\begin{remark}
If $\{X_t\}_{0\leq t\leq \epsilon}$ is a family of compact subanalytic subsets of $\mathbb{R}^m$ such that each $X_t$ is $C$-LNE and $\lim X_t=X_0$ (with respect to Hausdorff limit) then there exists $\delta \in (0,\epsilon)$ such that for each $t\in (0,\delta]$ there exists an epimorphism $h_t\colon \pi_1(X_t)\to \pi_1(X_0)$.
\end{remark}

\begin{example}
 It is easy to show that the Theorem \ref{Fundamental_thm} does not hold true if we remove $C$-LNE hypothesis.
 \begin{itemize}
  \item [(a)] For instance, we can obtain a family $\{S_t\}_{0\leq t\leq \epsilon}$ of compact subanalytic subsets of $\mathbb{R}^2$ such that for each $t>0$, $S_t$ is homeomorphic to $\mathbb{S}^1$, $S_0$ is homeomorphic to the wedge of two circles and $\lim S_t=S_0$ (with respect to the Hausdorff limit) and, thus, there is no epimorphism $h\colon \pi_1(S_t)\to \pi_1(S_0)$ for any $t>0$ (see Figures \ref{fig1a} and \ref{fig1b});
  \item [(b)] We can also obtain a family $\{T_t\}_{0\leq t\leq \epsilon}$ of compact subanalytic subsets of $\mathbb{R}^3$ such that for each $t>0$, $T_t$ is homeomorphic to $\mathbb{S}^1\times \mathbb{S}^1$, $T_0$ is homeomorphic to the wedge of two pinched torus and $\lim T_t=T_0$ (with respect to the Hausdorff limit) and, thus, there is no epimorphism $h\colon \pi_1(T_t)\to \pi_1(T_0)$ for any $t>0$ (see Figures \ref{fig2a} and \ref{fig2b});
 \end{itemize}
\end{example}
\begin{figure}[H]
  \begin{subfigure}{0.45\textwidth}
    \centering \includegraphics[scale=0.35]{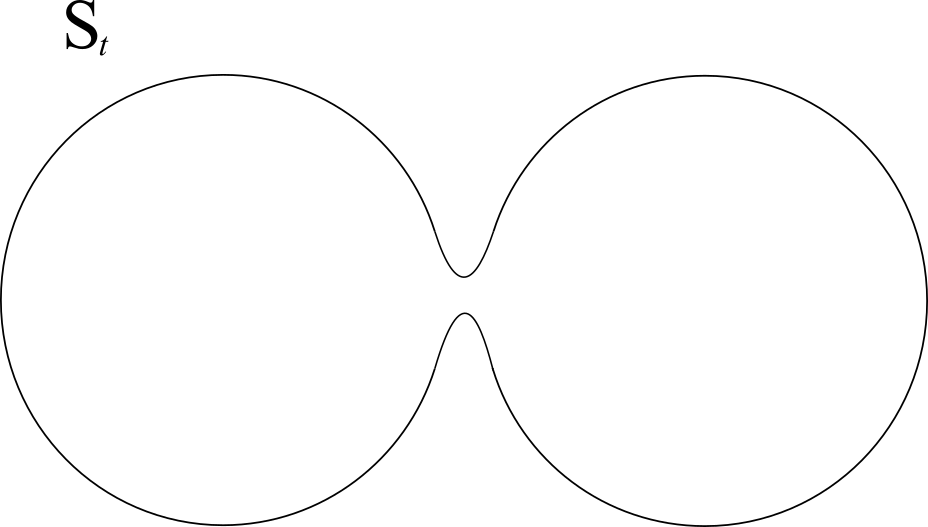}
    \caption{Topological circle}\label{fig1a}
  \end{subfigure}
\hfill
  \begin{subfigure}[c]{0.45\textwidth}
    \centering \includegraphics[scale=0.35]{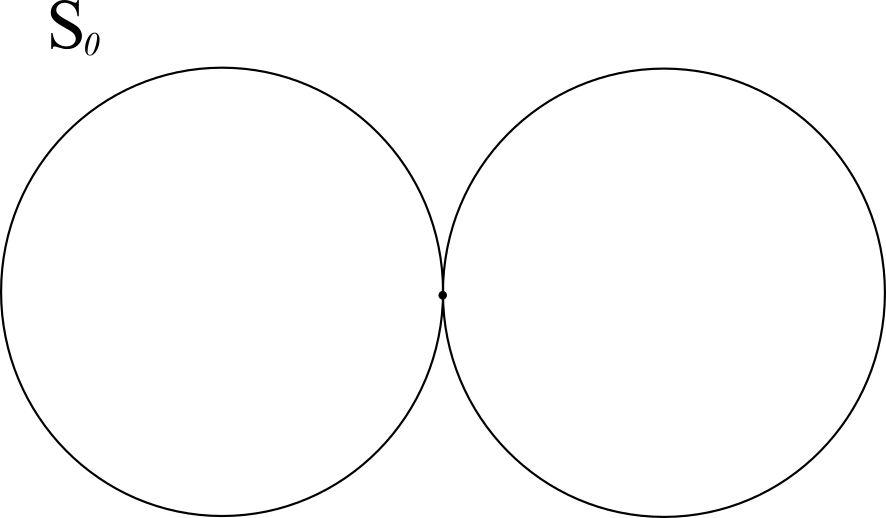}
    \caption{Wedge of two circles}\label{fig1b}
  \end{subfigure}
  \caption{Family of sets in $\mathbb{R}^2$}
\end{figure}

\begin{figure}[H]
  \begin{subfigure}[b]{0.49\textwidth}
    \centering \includegraphics[scale=0.6]{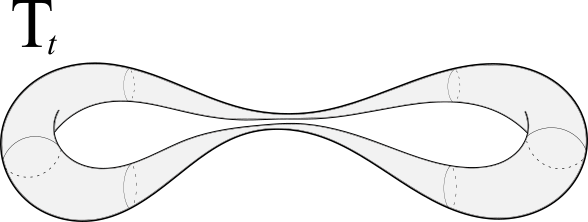}
    \caption{Topological torus}\label{fig2a}
  \end{subfigure}
  \begin{subfigure}[b]{0.49\textwidth}
    \centering \includegraphics[scale=0.6]{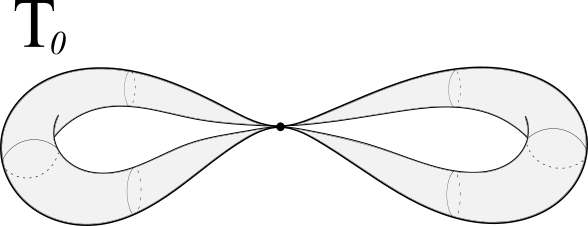}
    \caption{Wedge of two pinched torus}\label{fig2b}
  \end{subfigure}
  \caption{Family of sets in $\mathbb{R}^3$}
\end{figure}

We obtain the following important consequence.
\begin{corollary}\label{Fundamental_thm_local}
Let $X\subset \mathbb R^m$ be a closed subanalytic set. Assume that $X$ is LLNE at 0. Then there exists an epimorphism $h\colon \pi_1(L_0(X))\to \pi_1(X_0)$, where $X_0=C(X,0)\cap\mathbb{S}^{m-1}$.
\end{corollary}
\begin{proof}
Let $X_t:=(\frac{1}{t} X)\cap \mathbb S^{m-1}$ for $t>0 $.
Since $X$ is LLNE at $0$, there exist $C\geq 1$ and $\varepsilon>0$ such that $d_{X_t}(x,y)\leq C\|x-y\|$ for all $x,y\in X_t$ and for all $t\in (0,\varepsilon]$. By shrinking $\varepsilon$, if necessary, we assume that $\varepsilon$ is a Milnor radius for $X$ at $0$.
\begin{claim}\label{lne_set_lne_cone}
$X_0$ is $C$-LNE.
\end{claim}
\begin{proof}
We claim that for each pair of points $v,w\in X_0$ there exists an arc $\alpha\colon [0,1]\rightarrow X_0$ connecting $v$ to $w$ such that $length(\alpha)\leq C \|v-w\|$. In fact, let $v,w\in X_0$ and let $\{x_n\}_n$ and $\{y_n\}_n$ be two sequences of points in $X$ such that $\|x_n\|=\|y_n\|=t_n\to 0$ and $$\frac{1}{t_n}x_n\to v \quad \quad  \mbox{and} \quad \quad \frac{1}{t_n}y_n\to w.$$ For each $n$, let $\gamma_n\colon[0,1]\rightarrow X_{t_n}=X\cap \mathbb{S}_{t_n}^{m-1}$ be a geodesic arc on $X_{t_n}$ connecting $x_n$ to $y_n$. Let us define $\alpha_n\colon[0,1]\rightarrow \R^m$ by $t_n\alpha_n=\gamma_n$. We have seen that $\alpha_n$ is contained in the compact set $\mathbb{S}^{m-1}$. Moreover,
\begin{eqnarray*}
length(\alpha_n)&=&\frac{1}{t_n}length(\gamma_n) \\
&=& \frac{1}{t_n}d_{X_{t_n}}(x_n,y_n) \\
&\leq &\frac{C}{t_n} \|x_n-y_n\|.
\end{eqnarray*}
Since $\displaystyle\frac{1}{t_n} \| x_n-y_n\| \to \| v-w\|$, we get that $length(\alpha_n)$ is uniformly bounded and, therefore, by using Arzel\'a-Ascoli Theorem, there exists a continuous arc $\alpha\colon [0,1]\rightarrow X_0$ for which $\alpha_n$ converges uniformly, up to subsequence; let us say $\alpha_{n_k}\rightrightarrows\alpha$. Now, let us estimate length of $\alpha$. Given $\epsilon>0$, there is a positive integer number $k_0$, such that, for any $k>k_0$, $$ length(\alpha_{n_k})\leq C\|v-w\|+\epsilon, $$ hence $length(\alpha)$ is at most $C\| v-w\|+\epsilon.$ Since $\epsilon>0$ was arbitrarily chosen, we conclude that $length(\alpha)$ is at most $C \| v-w\|$. Therefore  $d_{X_0}(v,w)\leq  C\|v-w\|$ and this finishes the proof of Claim \ref{lne_set_lne_cone}.
\end{proof}

Thus, since $X$ is LLNE at $0$, it follows from Claim \ref{lne_set_lne_cone} that there exists a constant $C\geq 1$ such that $X_t$ is $C$-LNE for all $t\in [0,\varepsilon]$.
Moreover, $X_t\to X_0$ in the Hausdorff distance, then there exits $0<t_0\leq \varepsilon$ such that $dist_H(X_t,X_0)<\frac{\varepsilon_0(X_0)}{20 C^2}$ for all $t\leq t_0$.
Therefore, by Theorem \ref{Fundamental_thm}, there exists an epimorphism $h\colon\pi_1(X_{t_0})\to \pi_1(X_0)$ and since $t_0$ is a Milnor radius for $X$ at $0$, the proof is finished.
\end{proof}

\begin{remark}
Let us mention that for a complex analytic set $X$, $\pi_1(L_0(X))$ can be any finitely presented group. In fact, for any finitely presented group $G$, there exists  even $3$-dimensional complex analytic set $X$ with isolated singularity such that $\pi_1(L_0(X))\cong G$ (see Theorem 1 in \cite{KapovichK:2014}).
\end{remark}

In a similar way, we obtain also the following result.
\begin{corollary}\label{Fundamental_thm_infinity}
Let $X\subset \mathbb R^m$ be a closed semialgebraic set. Assume that $X$ is LLNE at infinity. Then there exists an epimorphism $h\colon \pi_1(L_{\infty}(X))\to \pi_1(X_{\infty})$, where $X_{\infty}=C_{\infty}(X)\cap\mathbb{S}^{2m-1}$.
\end{corollary}

An easy application of Corollary \ref{Fundamental_thm_infinity} is the following.
\begin{example}
The set $F=\{(x,y,z)\in \mathbb{C}^3;\, x^2+y^2+z^2=1\}$ is not LLNE at infinity, since by Kato-Matsumoto's Theorem \cite{KatoM:1975} $\pi_1(L_{\infty}(F))=0$ and by Mumford's Theorem $\pi_1(L_0(X))\not=0$, where $X=C_{\infty}(F)=\{(x,y,z)\in \mathbb{C}^3;\, x^2+y^2+z^2=0\}$.
\end{example}

\section{On characterization of smoothness}\label{sec:main_results}

Let us start this section stating our main result on characterization of smoothness of complex analytic sets, Theorem \ref{main_result}. We also list a series of consequences of Theorem \ref{main_result}.

\begin{theorem}\label{main_result}
Let $X\subset \mathbb C^m$ be an LNE complex analytic set at 0 and $k=\dim _{\mathbb{C}}X>1$. Assume that $\pi_1(L_0(X))\cong 0$. In this case, the following statements are mutually equivalent:
\begin{itemize}
\item [(1)] $X$ is an LLC $2k$-homology manifold at $0$;
\item [(2)] $X$ has no choking cycles at $0$ and $\pi_j(L_0(X))\cong 0$ for all $2\leq j\leq 2k-2$;
\item [(3)] $H_j(X_0)\cong 0$ for all $2\leq j\leq 2k-2$, where $X_0=C(X,0)\cap\mathbb{S}^{2m-1}$;
\item [(4)] $X$ is smooth at $0$.
\end{itemize}
In particular, if any of the above items holds true, we get that $X_t:=(\frac{1}{t} X)\cap \mathbb S^{2m-1}$ is diffeomorphic to $\mathbb S^{2k-1}$ for all small enough $t>0$.
\end{theorem}
\begin{proof}

First of all, it is clear that $(4)\Rightarrow (1)-(3)$. Since $X$ LNE at 0 and $\pi_1(L_0(X))\cong 0$, by Theorem 4.1 in \cite{MendesS:2021}, $X$ is LLNE at $0$.

Let $X_0=C(X,0)\cap\mathbb{S}^{2m-1}$ and $X_t:=(\frac{1}{t} X)\cap \mathbb S^{2m-1}_t$ for $t>0 $. We know that $X_t \to X_0$ in the Hausdorff distance. So, we are going to prove that $(1)\Rightarrow (3)$. Suppose that $X$ is an LLC $2k$-homology manifold.

To prove (3), we first prove that for each $p\in X$, there exists $\delta>0$ such that $X\cap B_{\delta}(p)$ is locally Ahlfors $2k$-regular.
In order to do this, let us remind that for each $x\in X$, the density of $X$ at $x$ is equal to the multiplicity of $X$ at $x$ (see \cite{Draper:1969});
$$
m(X,x)=\lim\limits_{t\to 0^+}\frac{\mathcal{H}_{2k}(X\cap B_t(x))}{\mu_{2k}t^{2k}},
$$
where $\mu_{2k}$ is the volume of the $2k$-dimensional unit ball.
Thus, we can choose $\delta>0$ such that $C:=\mathcal{H}_{2k}(X\cap B_{\delta}(p))<+\infty$. Moreover, since the density of $X$ is an upper-semi-continuous function (see \cite[Corollary 17.8]{Simon:1983}), we can assume that $m(X,p)\geq m(X,x)$ for all $x\in X\cap B_{\delta}(p)$.

Thus, for a compact subset $K\subset X\cap B_{\delta}(p)$, we define
$$r_K:=dist(K, \mathbb{C}^m\setminus X\cap B_{\delta}(p)).$$
Thus, if $r<r_K$ and $x\in K$, we have that $B_{r}(x)\subset B_{\delta}(p)$ and by Monotonicity Formula (see \cite[Proposition 1, p. 189]{Chirka:1989} or in \cite[identity 17.5, p.84]{Simon:1983}), we obtain
$$
1\leq \frac{\mathcal{H}_{2k}(X\cap B_r(x))}{\mu_{2k}r^{2k}}\leq \frac{\mathcal{H}_{2k}(X\cap B_{r_K}(x))}{\mu_{2k}r_K^{2k}}\leq \frac{C}{\mu_{2k}r_K^{2k}}.
$$
Therefore, by defining $C_K:=\frac{C}{\mu_{2k}r_K^{2k}}$, we obtain that $X\cap B_{\delta}(p)$ is locally Ahlfors $2k$-regular.

Since $X$ has a Lipschitz stratification (see \cite{Mostowski:1985}), we obtain also that $X$ is $2k$-rectifiable and, in particular, $X$ is metrically $2k$-dimensional.

Therefore, by Lemma \ref{heinonen}, $C(X,p)$ is a homology $2k$-manifold.
By using homological exact sequences, it is easy to see that $H_j(C(X,0)\setminus\{0\})=H_j(X_0)=0$,  for $1\leq j\leq 2k-2$.

Now, we are going to prove that $(2)\Rightarrow (3)$.
Let us suppose that $X$ has no choking cycles at $0$. Under such assumptions, we claim that the Hausdorff convergence $X_t\to X_0$ has the following property: fixed the integer $1\leq r\leq 2k-2$, for each $\epsilon>0$, there exist $\Delta >0$ and $\delta>0$ such that, for all $t<\delta$, any $r$-dimensional cycle $\xi_t$ in $X_t$ of diameter smaller than $\Delta$ bounds a chain $\eta_t$ in $X_t$ of diameter smaller than $\epsilon$. Indeed,  if not, then there exists $\epsilon>0$; for all $\Delta>0$ and $t>0$ there exists a $r$-dimensional cycle $\xi_t$ in $X_t$ of diameter smaller than $\Delta$ which bounds a chain $\eta_t$ in $X_t$ of diameter larger than or equal to $\epsilon$. It gives us a family of $r$-dimensional choking cycles of $X$ at $0$.

Once proved the claim above, we rely on Theorem 1.3 in \cite{Whyburn:1935} to get $H_j(X_0)\cong 0$ for all $2\leq j\leq 2k-2$.


In the next, we are going to prove that $(3)\Rightarrow (4)$. Thus, we assume that $H_j(X_0)\cong 0$ for all $2\leq j\leq 2k-2$. By Corollary \ref{Fundamental_thm_local}, $\pi_1(X_0)=0$ and by Hurewicz's Theorem (see \cite{Hurewicz:1935}), we have $\pi_j(X_0)=\pi_j(C(X,0)\setminus \{0\})=0$ for $j=1,...,2k-2$.
By the Theorem of Prill (Theorem \ref{prill}), we have that the reduced cone $C(X,0)$ is a linear subspace of $\mathbb{C}^m$.
Thus, we can consider the orthogonal projection
$$P\colon \mathbb{C}^n\rightarrow C(X,0).$$

Notice that the germ of the restriction of the orthogonal projection $P$ to $X$ has the following properties:
\begin{enumerate}
\item [(i)] The germ at $0$ of $P_{| X}\colon X\rightarrow C(X,0)$ is a ramified cover and the ramification locus is the germ of a codimension $1$ complex analytic subset $\Sigma$ of the plane $C(X,0)$;
\item [(ii)] The multiplicity of $X$ at $0$ can be interpreted as the degree $d$ of this germ of ramified covering map, i.e. there are open neighborhoods $U_1$ of $0$ in $X$ and $U_2$ of $0$ in $C(X,0)$, such that $d$ is the degree of the topological covering:
$$P_{| X}\colon X\cap U_1\setminus P_{| X}^{-1}(\Sigma)\rightarrow C(X,0)\cap U_2\setminus \Sigma;$$
\item [(iii)] There is $C>0$ such that $\|z\|\leq C\|P(z)\|$, for all $z$ enough small in $X$;
\item [(iv)] If $\gamma :[0,\varepsilon)\rightarrow X$ is a real analytic arc, such that $\gamma(0)=0$, then the  arcs $\gamma$ and $P\circ\gamma$ are tangent at $0$.

\end{enumerate}

Let us suppose that the degree $d$ is bigger than 1. Since $\Sigma$ is a codimension $1$ complex analytic subset of the space $C(X,0)$, there exists a unit tangent vector $v_0\in C(X,0)\setminus C(\Sigma,0)$ with $\|v_0\|=1$,
where $C(\Sigma,0) $ is the tangent cone of $\Sigma$ at $0$. Since $v_0$ is not tangent to $\Sigma$ at $0$, there exist positive real numbers $k$ and $\varepsilon $ such that the real cone
$$C_{k,\varepsilon }:=\{v\in C(X,0)\setminus \{0\};\ \|v\|\leq \varepsilon \mbox{ and }\|v-tv_0\|< kt \ \forall \  0<t<1\}$$
is a subset of $C(X,0)\setminus \Sigma.$
Let $V=P^{-1}(C_{k,\varepsilon })$ and since we have assumed that the degree $d\geq 2$, we have two (different) liftings $\gamma_1(t)$ and $\gamma_2(t)$ of the half-line $r(t)=tv_0$,  i.e. $P(\gamma_1(t))=P(\gamma_2(t))=tv_0$.   Since $P$ is the orthogonal projection on the reduced tangent cone $C(X,0) $, the vector $v_0$ is the unit tangent vector to the arcs $\gamma_1$ and $\gamma_2$ at $0$. Let $\alpha_i$ be a re-parametrization of $\gamma_i$ such that $\|\alpha_i(t)\|=t$, for all $t$ enough small, for $i=1,2$.

On the other hand, $\beta_t:[0,1]\to X$ being a rectifiable path connecting $\alpha_1(t)$ to $\alpha_2(t)$, then  $length (P\circ \beta_t)\geq \frac{k}{C}t$, since $\beta_t(0)$ and $\beta_t(1)$ belong to different connected components of $V$. It implies that $d_{X}(\alpha_1(t),\alpha_2(t))\geq \frac{k}{C}t$. But, since $\alpha_1(t)$ and $\alpha_2(t)$ are tangent at $0$, that is $\displaystyle\frac{\|\alpha_1(t)-\alpha_2(t)\|}{t}\to 0 \ \mbox{as} \ t\to 0^+,$ and $\frac{k}{C}>0$, we get a contradiction to the fact that $X$ is LLNE at $0$. Thus, $d=1$ and, therefore, $X$ is smooth at $0$.
\end{proof}

For sets with isolated singularity, we have the following:
\begin{theorem}\label{main_consequence}
	Let $X\subset\mathbb{C}^n$ be a complex analytic set of dimension $k$ with isolated singularity at $0$. Then, $X$ is smooth at $0$ if and only if $X$ is locally metrically conical at $0$ and its link at $0$ is $(2k-2)$-connected.
\end{theorem}
\begin{proof}
It is clear that if $X$ is smooth at $0$ then $X$ is locally metrically conical at $0$ and its link at $0$ is $(2k-2)$-connected.

Let us prove the converse. Assume that $X$ is locally metrically conical at $0$ and its link at $0$ is $(2k-2)$-connected. Then $X$ is LNE at $0$. Indeed, if $X$ is not LNE at $0$, by Arc Criterion Theorem (see \cite[Theorem 2.2]{BirbrairM:2018}), there are a $v\in C(X,0)\cap \mathbb{S}^{2n-1}$ and two subanalytic arcs $\gamma_1,\gamma_2\colon [0,\varepsilon)\to X$ such that $\gamma_i(t)=tv+o(t)$ and $\|\gamma_i(t)\|=t$ for all $t\in [0,\varepsilon)$ and $i=1,2$ such that 
$$
\lim \limits_{t\to 0^+}\frac{d_X(\gamma_1(t), \gamma_2(t))}{\|\gamma_1(t)- \gamma_2(t)\|}=+\infty.
$$

Since $X$ is locally metrically conical at $0$, there are $r,\delta>0$ and a bi-Lipschitz homeomorphism $\varphi\colon X\cap C_{r}(v)\cap \overline{B_{\delta}(0)}\to Cone_1(X\cap C_{r}(v)\cap \mathbb{S}^{2n-1}_{\delta})$ such that $\|\varphi(x)\|=\|x\|$ for all $x\in X\cap C_{r}(v)\cap \overline{B_{\delta}(0)}$, where $C_{r}(v)=\{tw; w\in \overline{B_r(v)}$ and $t\in [0,1]\}$. Then, by Theorem 3.2 in \cite{Sampaio:2016}, there is a bi-Lipschitz homeomorphism $\psi:=d_0\varphi\colon C(X,0)\cap C_{r}(v)\cap \overline{B_{\delta}(0)}\to Cone_1(X\cap C_{r}(v)\cap \mathbb{S}^{2n-1}_{\delta}) $. Since $X$ has an isolated singularity at $0$, by shrinking $\delta$, if necessary, we can assume that $X_t:=X\cap \mathbb{S}^{2n-1}_t$ is a compact submanifold for all $0<t\leq \delta$. Let $U\subset C_{r}(v)$ be an open neighborhood of $w=\psi (\delta v)$ such that $U\cap X\cap \mathbb{S}^{2n-1}_{\delta}$ is LNE. Since $\gamma_1(t)$ and $\gamma_2(t)$ are tangent to $v$ at $t=0$, $\gamma_i(t)\in A:=\varphi^{-1}(Cone_1(U\cap X_{\delta}))$ for all small enough $t>0$ and $i=1,2$. However, $A$ is LNE, and thus there exists a constant $C\geq 1$ such that $d_A(\gamma_1(t),\gamma_2(t))\leq C\|\gamma_1(t)-\gamma_2(t)\|$ for all small enough $t>0$. Since $d_X(\gamma_1(t),\gamma_2(t))\leq d_A(\gamma_1(t),\gamma_2(t))$, we obtain a contradiction. Therefore, $X$ is LNE at $0$.

%

Since $X$ is locally metrically conical $0$, it has no choking cycles at $0$. Thus, $X$ is LNE at $0$, has no choking cycles at $0$ and $\pi_j(L_0(X))\cong 0$ for all $1\leq j\leq 2k-2$. By item (2) of Theorem \ref{main_result}, $X$ is smooth at $0$.
\end{proof}

\begin{corollary}\label{mumford_gen_curvature}
Let $X\subset\mathbb{C}^m$ be a complex analytic set with an isolated singularity at $0\in X$ and $k=\dim X>0$. Assume that $\pi_i(L_0(X))=0$ for all $1\leq i\leq 2k-2$ and $X$ is LNE at $0$. If there exists $K_0\in [0,+\infty)$ such that the sectional curvature of $X_t:=(\frac{1}{t}X)\cap \mathbb{S}^{2m-1}$, denoted by $K_{X_t}$, satisfies $0\leq |K_{X_t}|\leq K_0$, for all small enough $t>0$, then $X$ is smooth at $0$. In particular, we obtain that the link of $X$ at $0$ is diffeomorphic to $\mathbb S^{2k-1}$.
\end{corollary}
\begin{proof}
By coarea formula, there exists $t_0>0$ such that $\mathcal H_{2k-1}(X_t)\geq \frac{1}{2}$, for all $0<t\leq t_0$, since $\lim\limits_{t\to 0^+}\frac{\mathcal{H}_{2k}(X\cap B_t(0))}{\mu_{2k}t^{2k}}=m(X,0)\geq 1$, where $\mu_{2k}$ is the Hausdorff measure of the $2k$-dimensional unit ball.
Thus, by Theorem 2.1 in \cite{Cheeger:1970}, there exists a constant $r>0$ that does not depend on $t$ such that for each $t$ and each $p\in X_t$ the injectivity radius of the exponential map at $p$ is greater than $r$. This together with the fact that $X$ is LNE (at $0$) imply that $X$ has no choking cycles at $0$. Thus, by Theorem \ref{main_result}, $X$ is smooth at $0$.
\end{proof}

Since any complete intersection $X$ with $\dim X>2$ has simply connected link (see \cite{Hamm:1971}, cf. p. 2 in \cite{Massey}), we obtain the following version of the A'Campo-L\^e's Theorem.
\begin{corollary}\label{ACampo-Le_gen}
Let $X\subset\mathbb{C}^m$ be a complex analytic set which is a complete intersection with $k=\dim X>2$. Then $X$ is smooth at $0$ if and only if $X$ is LLC, $(2k)$-homology manifold and LNE at $0$.
\end{corollary}

\begin{remark}
In order to obtain smoothness at $0$ of a complex analytic curve $X\subset \mathbb{C}^n$, it is enough to ask that it is LNE at $0$ and $\pi_1(L_0(X))\cong\mathbb Z$.
\end{remark}

\begin{remark}\label{rem:trivial_cone}
If $X\subset \mathbb C^m$ is a $k$-dimensional complex analytic set with no choking cycles at $0$ and $H_j(L_0(X))= 0$ for all $1\leq j\leq 2k-2$ such that $\pi_1(L_0(C(X,0)))= 0$ then by looking at the proof of Theorem \ref{main_result}, we obtain that $C(X,0)$ is a linear subspace, without the hypothesis that $X$ is LNE at $0$;
\end{remark}

Let us remark that there exist examples of complex analytic sets which are topological manifolds but they are not LLC.
\begin{example}
For each integer $k>1$, the set $X=\{(x_1,...,x_{2k})\in \mathbb{C}^{2k};\, x_1^2+...+x_{2k-1}^2=x_{2k}^3\}$ is a topological manifold (see \cite{Brieskorn:1966}), but it is not smooth. Moreover, since $C(X,0)$ is not a linear subspace and $\pi_1(L_{C(X,0)})=0$, by the proof of Theorem \ref{main_result}, $X$ cannot be LLC.
\end{example}

As another consequence, we obtain the following.


\begin{corollary}[Theorem 4.2 in \cite{Sampaio:2016}]
Let $X\subset\mathbb{C}^m$ be a $k$-dimensional complex analytic set. If $X$ is Lipschitz regular, then $X$ is smooth.
\end{corollary}

Moreover, it follows from the proof of Theorem \ref{main_result} the following result.

\begin{corollary}\label{mumford_gen}
Let $X\subset\mathbb{C}^m$ be a complex analytic set with $k=\dim X>0$. Then $X$ is smooth at 0 if and only if $\pi_1(L_0(X))\cong \pi_1(\mathbb S^k)$, $X$ is LNE at $0$ and $H_i(L_{C(X,0)})=0$ for $2\leq i\leq 2k-2$.
\end{corollary}

Let us analyze the hypotheses of Corollary \ref{mumford_gen}. First, let us observe that we cannot remove the condition that $X$ is LNE at $0$.
\begin{example}\label{non_LLNE}
The set $X=\{(x,y,z)\in \mathbb{C}^3;\, y^2-x^3=0\}$ satisfies $\pi_1(L_0(X))=0$ and $H_i(L_{C(X,0)})=0$ for $2\leq i\leq 2$, but $X$ is not smooth at $0$.
\end{example}
Now, let us show that we cannot also remove the condition $H_i(L_{C(X,0)})=0$ for $2\leq i\leq 2k-2$.
\begin{example}\label{non_hm}
The set $X=\{(x,y,z,w)\in \mathbb{C}^4;\, x^2+y^2+z^2=0\}=\{(x,y,z)\in \mathbb{C}^3;\, x^2+y^2+z^2=0\}\times \mathbb{C}$ is LNE (see Propositions 2.6 and 2.8 in \cite{KernerPR:2018}) and $\pi_1(L_0(X))=0$ (see Theorem 5.2 in \cite{Milnor:1968}), but $X$ is not smooth at $0$.
\end{example}



Let us remark also that Theorem \ref{main_result} and Corollaries \ref{mumford_gen} and \ref{ACampo-Le_gen} do not hold true in the real case.

\begin{example}[Proof of Proposition 1.9 in \cite{Sampaio:2019}]\label{non_complex}
The homogeneous algebraic set $X=\{(x_1,x_2,x_3,x_4)\in \R^4;\, x_1^{2021}+x_2^{2021}=x_3^{2021}\}$ is subanalytic bi-Lipschitz homeomorphic to $\R^3$ and, in particular, $X$ is LNE at 0, LLC, a $3$-homology manifold and $\pi_1(L_0(X))=0$, but it is not smooth.
\end{example}

\begin{corollary}\label{main_result_infinity}
Let $X\subset \mathbb C^m$ be a complex algebraic set which is LLNE at infinity and $k=\dim _{\mathbb{C}}X>1$. Assume that $\pi_1(L_{\infty}(X))=0$. In this case, the following statements are mutually equivalent:
\begin{itemize}
\item [(1)] $X$ has no choking cycles at infinity and $H_j(L_{\infty}(X))= 0$ for all $2\leq j\leq 2k-2$;
\item [(2)] $H_j(X_{\infty})= 0$ for all $2\leq j\leq 2k-2$, where $X_{\infty}=C_{\infty}(X)\cap\mathbb{S}^{2m-1}$;
\item [(3)] $X$ is an affine linear subspace.
\end{itemize}
In particular, if any of the above items holds true, we get that for each $p\in X$, $X_t:=[\frac{1}{t} (X-p)]\cap \mathbb S^{2m-1}$ is isometric to $\mathbb S^{2m-1}$ for all $t>0$.
\end{corollary}
\begin{proof}
It is easy to see that $(3)\Rightarrow (1)$ and $(2)$. Moreover, by similar arguments as in the proof of Theorem \ref{main_result}, we prove $(1)\Rightarrow (2)$ and by showing that the degree of $X$ is 1, like we did to prove that the multiplicity is 1 in the proof of Theorem \ref{main_result}, we prove $(2)\Rightarrow (3)$.
%
%
\end{proof}

\begin{remark}
If $X\subset \mathbb C^m$ is a $k$-dimensional complex algebraic set with no choking cycles at infinity and $H_j(L_{\infty}(X))= 0$ for all $1\leq j\leq 2k-2$ such that $\pi_1(L_0(C_{\infty}(X)))= 0$ then likewise Remark \ref{rem:trivial_cone} we get that $C_{\infty}(X)$ is a linear subspace, without the hypothesis that $X$ is LLNE at infinity;
\end{remark}

\begin{example}
The set $X=\{(x,y,z)\in \mathbb{C}^3;\, z=xy\}$ has choking cycles at infinity, since $\pi_j(L_{X,\infty})= 0$ for $j\in\{1,2\}$, $\pi_1(L_0(C_{\infty}(X)))= 0$ and $X$ is not a plane.
\end{example}

\end{document}